\documentclass{amsart}
  \usepackage{macros_anglais,macros_anglais_format,macros_anglais_geodiff}

\title{Almost homogeneous manifolds with boundary}
\author{Beno\^{\i}t Kloeckner}

\newcommand{\moduli}{\mathop {\mathscr{M}}\nolimits}
\newcommand{\stab}{\mathop {\mathrm{Stab}}\nolimits}

\begin{document}
%%%%%%%%%%%%%%%%%%%%%%%%%%%%%%%%%%%%%%%%%%%%%%%%%%%%%%%%%%%%%%%
%%%%%%%%%%%%%%%%%%%%%%%%%%%%%%%%%%%%%%%%%%%%%%%%%%%%%%%%%%%%%%%

\maketitle

%%%%%%%%%%%%%%%%%%%%%%%%%%%%%%%%%%%%%%%%%%%%%%%%%%%%%%%%%%%%%%%
\section{Introduction}

In this article we study the classification of some 
Lie group actions on differentiable manifolds, up to
conjugacy. Let us start with the precise notion of conjugacy
we shall use.

Two $\diffb{r}$ actions $\rho_1$, $\rho_2$ of a Lie group $G$ on a manifold $M$
are said to be $\diffb{r}$ \defini{conjugate} if there is a $\diffb{r}$ diffeomorphism
$\Phi : M\to M$ that is $(\rho_1,\rho_2)$ equivariant, that is:
\[\Phi(\rho_1(g)x)=\rho_2(g)\Phi(x)\qquad\forall x\in M\,\forall g\in G.\]
A conjugacy class will be denoted with brackets: $[\rho_i]$.

These two actions are said to be \defini{topologically conjugate} if $\Phi$ is only
a homeomorphism.

When $\Phi$ is both a $\diffb{r}$ map 
and a homeomorphism (but not necessarily a diffeomorphism), we say that $\rho_1$
is $\diffb{r}$ semi-conjugate to $\rho_2$.

Note that $r$ is either a positive integer, $\infty$ or $\omega$ ($\diffb{\omega}$
meaning real analytic), and that $M$ is possibly with boundary.

Given a $\diffb{r}$ action $\rho_0$, we are interested in the quotient set 
$\moduli^r(\rho_0)$ of all actions that are topologically conjugate
to $\rho_0$, modulo $\diffb{r}$ conjugacy. Can we entirely describe $\moduli^r(\rho_0)$,
determine its size, endow it with some natural structure ?

For a general action these questions may be too wide to be answered, and
we shall tackle specific cases. The case of transitive actions is easily dealt with: 
the manifold is then a homogeneous space, and it is sufficient
to give the stabilizer of a point to completely describe the action. The
set $\moduli^r(\rho_0)$ is therefore reduced to $[\rho_0]$.

As soon as there are two orbits, the question is not so simple.
Not only does one need a stabilizer for each orbit, but several 
non differentiably conjugate actions can have the same pair of 
stabilizers. Let us consider the ``differentiable
compactifications'' of rank one symmetric spaces.

A \defini{differentiable compactification} of a non-compact homogeneous 
space $X=G/H$ is an identification of $X$ with the interior of a manifold 
with boundary on which the action of $G$ extends differentiably (see Definition \ref{diff_comp}).
Differentiability can be replaced by other regularity assumptions, mainly
smoothness ($\diffb{\infty}$) and analyticity ($\diffb{\omega}$).
Every non-compact rank one symmetric space $\KH^n$ admits an analytic compactification,
obtained from its Klein ball model, and called its projective compactification.
Since such a space is an isotropic Riemannian manifold, the action of its
isometry group on this compactification has two orbits: one is the interior
$\KH^n$, the other is its geodesic boundary.

This particular case fits into the following framework. Let $\rho_0$ be a $\diffb{r}$ action 
of a connected Lie group $G$ on a
manifold with boundary $M$ of dimension $n$. Denote by $\Int(M)$ the interior
of $M$ and by $\partial M$ its boundary. When $\rho_0$ is transitive on 
$\Int(M)$, we say that $M$ (together with the action) is
an \defini{almost homogeneous manifold with boundary}. This property is assumed to hold from now on.
Moreover, we set $r=\infty$ or $\omega$.

For some $\diffb{r}$ functions $f:[0,1]\to[0,1]$
(the typical case being $f(y)=y^p$ with $p\in\mN$)
we define the ``stretch'' of $\rho_0$ by $f$, a new action of $G$ on $M$ that is topologically
conjugate to $\rho_0$ (see Proposition \ref{construction}, Definition \ref{stretch}
and Theorem \ref{regularity}). 

Our main result is a generalisation of that
of \cite{Kloeckner2}, where we described the case of the real hyperbolic spaces $\RH^n$.
\begin{theorem*}
If $r=\omega$, or $r=\infty$ and $\mK=\mH$ or $\mO$,
then any $\diffb{r}$ compactification of the non-compact
symmetric space $\KH^m$ is
a stretch of its projective compactification.
Moreover two different stretches give non-conjugate compactifications.
\end{theorem*}
In particular, this shows that if $\rho_0$ is the projective compactification 
of a rank one symmetric space, then $\moduli^\omega(\rho_0)$ is countably infinite.
Note that here, two stretches are considered to be different if they
are build from non-equivalent stretching function (see Definition \ref{stretch_func}).

The tools we use are roughly the same as in  \cite{Kloeckner2}, but we try to
apply them to the broader context of almost homogeneous manifolds
with boundary. This generalization is only partial: in particular, we need
to assume an algebraic condition (see Section \ref{condition_A})
to ensure that $\moduli^\omega(\rho_0)$ is infinite.

%%%%%%%%%%%%%%%%%%%%%%%%%%%%%%%%%%%%%%%%%%%%%%%%%%%%%%%%%%%%%%%
\section{Stretching an action}

A stretching of $\rho_0$ consists of gluing together
its two orbits in a new way. To achieve this, we need a function that indicates ``at what speed''
we glue them.

\begin{definition}\label{stretch_func}
A \defini{stretching function} is a function $f:[0,1]\to[0,1]$ such that
\begin{enumerate}
\item $f$ is an orientation-preserving homeomorphism,
\item $f$ is a $\diffb{r}$ function and the restriction of $f$ to $]0,1]$ is a $\diffb{r}$ diffeomorphism
\item the function $f/f'$, which is well defined on $]0,1]$, can be extended
to a $\diffb{r}$ function at $0$.
\end{enumerate}
We say that a stretching map $f$ is \defini{trivial} when it is a diffeomorphism (that is,
when $f'(0)\neq 0$), and that two stretching maps $f,g$ are \defini{equivalent} if $g^{-1}f$ is a 
diffeomorphism.
\end{definition}

The reason why we need the last condition will become clear in the proof of Theorem \ref{regularity}.
This condition is satisfied by any non-flat (that is: having non-trivial
Taylor series) function. In the real-analytic case, any stretching function is equivalent to $y\mapsto y^p$
for some $p\in\mN$.

Given a stretching function, we define a map $M\to M$ that will relate $\rho_0$ and its
stretching. Note that to simplify the construction of manifolds, we consider charts
with values not only in $\mR^n$, but in any manifold.

\begin{proposition}\label{construction}
Let $f$ be a stretching function. There is a homeomorphism $\Phi_f:M\to M$ that
is a $\diffb{r}$ map, whose restriction $\Int(M)\to\Int(M)$ is a $\diffb{r}$ diffeomorphism,
and such that for every point $p\in\partial M$ there exist two systems of coordinates 
$(x_1,\dots,x_{n-1},y)$ around $p$ and $\Phi_f(p)$, 
where $y$ is a locally defining function for $\partial M$ and 
such that in these coordinates
\[\Phi_f(x_1,\dots,x_{n-1},y)=(x_1,\dots,x_{n-1},f(y)).\]
\end{proposition}

Since the latter formula is valid in two different systems of coordinates on the domain and
the range of $\Phi_f$, this map should be considered as being defined up to composition with
diffeomorphisms on the left and right, and not up to conjugacy. This is natural since
we shall use it to pull back $\rho_0$, so the domain and range should be considered as
two different copies of $M$.

\begin{proof}
Let $f$ be a stretching map. Let $C$ be collar neighborhood of $\partial M$ in $M$,
parametrized by coordinates $(x,y)\in\partial M\times [0,1[$. Note that we ask that
a collar neighborhood has its complement diffeomorphic to $M$.

 Let 
$c:\partial M\times ]0,1[\to\Int(M)$ be the inclusion map (where $\partial M\times ]0,1[$
is identified with $C\cap\Int(M)$ \lat{via} the coordinates). Define a $\diffb{r}$ manifold with boundary 
$M'$ by two charts that are copies of $\Int(M)$ and $C$, with change of coordinates given by
\begin{eqnarray*}
 \partial M\times ]0,1[ &\to& \Int(M) \\
  (x,y) &\mapsto& c(x,f(y))
\end{eqnarray*}

\begin{figure}[htb]\begin{center}
\input{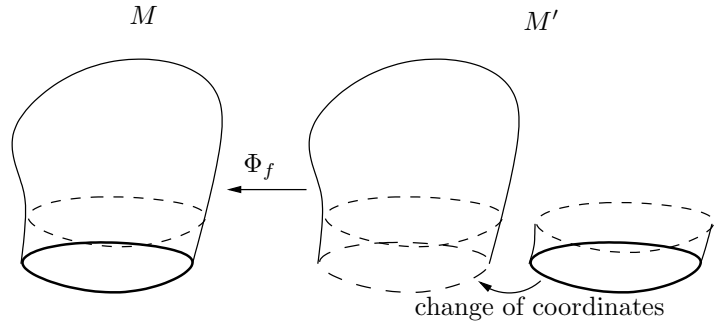}
\caption{Construction of a stretching map}
\end{center}\end{figure}

\begin{lemma}
If $r=\infty$, the resulting manifold $M'$ only
depends (up to $\diffb{\infty}$ diffeomorphism) upon the equivalence class of the stretching map.
\end{lemma}

\begin{proof}
First, the choice of neighborhood does not matter.

Now let $g$ be a stretching map equivalent to $f$ and denote by $M''$ the manifold
obtained from the two charts $\Int(M)$ and $C$ by the change of coordinates
\begin{eqnarray*}
 \partial M\times ]0,1[ &\to& \Int(M) \\
  (x,y) &\mapsto& c(x,g(y))
\end{eqnarray*}

Then the map $(x,y)\mapsto(x,g^{-1}f(y))$ is a diffeomorphism from a collar
neighborhood of $\partial M'$ to a collar neighborhood of $\partial M''$. Moreover, this map extends
to the image of the extensions $\tilde{c}(.,1)$ of $c$ that maps $\partial M\times{1}$ into the interiors $\Int(M'),\Int(M'')\simeq \Int(M)$.
Using a partition of unity,
one can extend further this map to a $\diffb{\infty}$ diffeomorphism between $M'$ and 
$M''$.
\end{proof}

\begin{lemma}
In the two cases $r=\infty,\omega$, the previous construction gives a manifold with boundary $M'$ that is
$\diffb{r}$ diffeomorphic to $M$, regardless the choice of the stretching function $f$.
\end{lemma}

\begin{proof}
First note that the result is obvious for the manifold 
with boundary $U\times[0,+\infty[$, where $U$ is any open set of $\mR^{n-1}$.

If $r=\infty$, the previous lemma enables
one to assume that $f$ is the identity on a neighborhood $]1-2\varepsilon,1]$ of $1$. 
Take then a covering of $\partial M$ by charts $(U_i)$:
the manifolds $M$ and $M'$ are covered by the same charts
$(U_i\times[0,1-\varepsilon[)$ and $\Int(M)\setminus \partial M\times]0,1-2\varepsilon[$,
with the same changes of coordinates. They are therefore $\diffb{\infty}$ diffeomorphic.

If $r=\omega$, the previous discussion shows that $M$ and $M'$
are $\diffb{\infty}$ diffeomorphic. But a given differentiable manifold admits 
only one real-analytic structure (this is discussed in more details
at the end of the section, see \ref{analytic_structure}), thus
$M$ and $M'$ are $\diffb{\omega}$ diffeomorphic.
\end{proof}

We have now two presentations of $M$ by the same pair of charts $C$, $\Int(M)$ but with
different changes of coordinates. The identity on $\Int(M)$ extends continuously
into a $\diffb{r}$ map $\Phi_f:M'\to M$. Up to composition by a diffeomorphism
$M\to M'$, $\Phi_f$ has the desired properties.
\end{proof}

\begin{definition}\label{stretch}
A map $\Phi_f$ satisfying the conclusion of Proposition \ref{construction} is said
to be a stretching map associated to the stretching function $f$.
A stretch of $\rho_0$ is an action $\rho_f=\Phi_f^*(\rho_0)$
where $\Phi_f$ is a stretching map.
\end{definition}

Let us prove that the conditions we put on $f$ ensure that a stretch $\rho_f$ is as regular
as $\rho_0$ ; recall that $\rho_0$ is an action of a \emph{connected}
Lie group $G$.

\begin{theorem}\label{regularity}
Any stretch of $\rho_0$ is a $\diffb{r}$ action on $M$.
\end{theorem}

\begin{proof}
We first go to the Lie algebra level.

Let $\al{g}$ be the Lie algebra of $G$. For any $X\in\al{g}$, we denote by
$\rho_0(X)$ the corresponding vector field on $M$. Then
$\phi_f^*(\rho_0(X))$ is well defined and $\diffb{r}$ on $\Int(M)$ since
by construction $\Phi_f$ is a diffeomorphism when restricted to $\Int(M)$.
Let us prove that it extends to a $\diffb{r}$ vector field on $M$.
This will be done by considering local charts near the boundary.

We start with the analytic case. Since any analytic stretching map
is equivalent to the lower order term of its Taylor series, we can assume that $f$ is
written in the form
$y\mapsto y^p$ for some integer $p$. In suitable charts $(x_i,y)$ 
near the boundary (locally defined by $y=0$), we can write 
$\Phi_f$ in the form $(x_i,y)\mapsto(x_i,y^p)$ and
$\rho_0(X)$ in the form
\[ \rho_0(X)=\sum_{i=1}^{n-1}\sum_{a,b}
     \alpha_{a,b}^i x^a y^b\frac{\partial}{\partial x_i}
  +\sum_{a,b}\beta_{a,b} x^a y^b \frac{\partial}{\partial y}\]
where the sums are taken over all non-negative integers $b$ 
and all $(n-1)$-tuples of non-negative integers $a$; $x^a$ means
$x_1^{a_1} x_2^{a_2}\dots x_{n-1}^{a_{n-1}}$.
By a direct computation, we see that 
\[ \Phi_f^*(\rho_0(X))=\sum_{i=1}^{n-1}\sum_{a,b}
     \alpha_{a,b}^i x^a y^{pb}\frac{\partial}{\partial x_i}
  +\sum_{a,b}\beta_{a,b} x^a y^{pb+1-p} \frac{\partial}{\partial y}\]
hence this map is analytic since, $\rho_0(X)$ being tangent to the boundary, 
$\beta_{a,0}=0$  for all $a$.

In the smooth case, we write
\[ \rho_0(X)=\sum_{i=1}^{n-1}\sum_{a,b}
     \alpha^i(x,y)\frac{\partial}{\partial x_i}
  +\beta(x,y) \frac{\partial}{\partial y}\]
where $\alpha_i$ and $\beta$ are $\diffb{\infty}$ functions, and $\beta(x,0)=0$ for all
$x$. According to the Hadamard lemma (see for example \cite{Arnold}), there is
a $\diffb{\infty}$ function $\beta_1$ such that $\beta(x,y)=\beta_1(x,y)y$.
Then we get
\[ \Phi_f^*(\rho_0(X))=\sum_{i=1}^{n-1}
     \alpha^i(x,f(y))\frac{\partial}{\partial x_i}
  +\beta_1(x,f(y))\frac{f(y)}{f'(y)} \frac{\partial}{\partial y}\]
and, since $f/f'$ extends differentiably, so does $\Phi_f^*(\rho_0(X))$.

We get that $\Phi_f^*(\rho_0)$ defines a $\diffb{r}$ action of
$\al{g}$ on $M$. Thus, it defines a $\diffb{r}$ action
of the universal covering $\wt{G}$ of $G$. But an element that
projects to $1\in G$ must act trivially on $\Int(M)$, and by continuity it
acts trivially on $M$. We thus get a $\diffb{r}$ action of $G$
on $M$ that coincides with $\Phi_f^*(\rho_0)$ on $\Int(M)$.
This last action is therefore $\diffb{r}$.
\end{proof}

The main reason for restricting ourselves to smooth and analytic actions is
the loss of regularity in the Hadamard lemma for $\diffb{r}$ functions,
$r<\infty$.

%%%
\subsection{Uniqueness of the analytic structure of a smooth manifold with 
boundary}\label{analytic_structure}

We used above the following result.
\begin{theorem}\label{analytic_theorem}
Let $M'$ and $M$ be two compact real-analytic manifolds with (analytic) boundary.
If there is a smooth diffeomorphism $F:M'\to M$, then there is also an
analytic one.
\end{theorem}

The without boundary version of this theorem is well known: Grauert proved
in \cite{Grauert} that the set of analytic diffeomorphisms between two
analytic manifold without boundary is dense into the set of smooth ones.

However, even if the ``with boundary'' version is unsurprising and cannot be
expected to be new, it is very difficult to find in the literature.
Luckily, it can be deduced from the following result and the Morrey-Grauert theorem
that states that any analytic manifold can be analytically embedded in $\mR^N$ for some
$N$.
\begin{theorem}[Tognoli \cite{Tognoli}]
Let $U$ be a relatively compact open set in $\mR^N$, $V$ a coherent analytic subset of $U$
and $\ell$ a smooth function on $U$ whose restriction to $V$ is analytic. Then for all
$k$ and all $\varepsilon>0$ there is an analytic function $h$ defined on $U$ such that
\begin{enumerate}
\item $\Vert \ell-h\Vert_k<\varepsilon$ where $\Vert\cdot\Vert_k$ is the $\diffb{k}$ norm,
\item $h$ and $\ell$ coincide on $V$
\end{enumerate}
\end{theorem}
Note that an analytic submanifold is a special case of a coherent analytic subset.

\begin{proof}[Proof of Theorem \ref{analytic_theorem}]
Up to doubling, one has to prove that if $M$ and $M'$ are compact
analytic manifolds without boundary,
$N$ and $N'$ are analytic submanifolds of $M$ and $M'$, and $F: M\to M'$
is a smooth diffeomorphism that restricts to a smooth diffeomorphism
$N\to N'$, then there is an analytic diffeomorphism $M\to M'$ that maps
$N$ onto $N'$.

Thanks to the Morrey-Grauert Theorem we can embed $M$ and $M'$ into $R^N$,
and using the Grauert Theorem we construct an analytic diffeomorphism $L : N \to N'$.
It extends into a smooth diffeomorphism $L : M \to M'$. Moreover, we can
smoothly extend $L$ to a relatively compact open set $U \subset R^N$ containing $M$.
  Since $M'$ is analytic, it admits a neighborhood $U'$ that retracts
analytically onto $M'$ (see \cite{Krantz-Parks} Theorem 2.7.10). 
Using the approximation theorem of Tognoli on the
coordinates of $L$, with $k=1$, we construct a map $H : M \to U'$ that is an
analytic diffeomorphism onto its image, coincides with $L$ on $N$, and is
$\diffb{1}$ close to $L$.
  Composing $H$ with the retraction $U' \to M'$, we get the desired analytic
diffeomorphism $(M,N) \to (M',N')$ (the $\diffb{1}$ closeness to $L$ ensures that 
this actually is a diffeomorphism).
\end{proof}

%%%%%%%%%%%%%%%%%%%%%%%%%%%%%%%%%%%%%%%%%%%%%%%%%%%%%%%%%%%%%%%
\section{Non-conjugacy of stretches}

Now, we would like to distinguish between the various stretches of $\rho_0$, in
order to ensure that $\moduli^r(\rho_0)$ is large. Unfortunately,
we are able to do so only under an algebraic assumption. 

Note that for many actions, explicit computations as in the proof of Theorem \ref{regularity}
will be sufficient to prove that $\Phi_f^*(\rho_0)$ is not conjugate to $\rho_0$. For example,
the valuation of the Taylor expansion of $\rho_0(X)$, where $X$ is any given element of $\al{g}$,
 is a conjugacy invariant. Problems are however expected when all elements of $\al{g}$ act
very flatly near the boundary.

%%%%%%%%%%%%%%%%%
\subsection{An algebraic condition}\label{condition_A}

Let $x$ be a point of the interior of $M$ and let $H\subset G$ be its stabilizer for the
action $\rho_0$. Denote by $\stab(H)$ the subgroup of elements $g\in G$ such that
$gHg^{-1}=H$, and by $Z(G)$ the center of $G$. The inclusion $\stab(H)\supset H\cdot Z(G)$
always holds; we consider the converse inclusion.
\begin{itemize}
\item[(A)] \quad $\stab(H) = H\cdot Z(G)$.
\end{itemize}
Note that (A) does not depend upon the choice of $x$. 

Let us stress two particular cases.
First, when  no two points of $\Int(M)$ have the same stabilizers, $\stab(H)=H$ and
(A) holds. For example, this is the case for symmetric spaces.
Second, when $G$ is abelian $\Stab(H)=G=Z(G)$.

Moreover (A) is stable by direct product in the following sense.
\begin{proposition}
Consider two actions $\rho_i$ of groups $G_i$ on manifolds $M_i$ ($i=1,2$),
one being a homogeneous manifold and the second an almost homogeneous manifold with boundary.
If (A) holds for both actions, then (A) holds for the action of $G_1\times G_2$
on $M_1\times M_2$ defined by
\[(g_1,g_2)\cdot(x_1,x_2)=(\rho_1(g_1)x_1,\rho_2(g_2)x_2).\]
\end{proposition}

\begin{proof}
With obvious notations we have $H=H_1\times H_2$, $Z(G)=Z(G_1)\times Z(G_2)$
and $\Stab(H)=\Stab(H_1)\times\Stab(H_2)$ and the result follows.
\end{proof} 

This applies for example to the product of the projective compactification
of a non-compact rank one symmetric space on the one hand, and a torus or a
compact symmetric space on the other hand.

Another example of an almost homogeneous manifold satisfying (A) is the Poin\-car\'e compactification
of Euclidean space. It is obtained as follows: one considers $\mR^n$ as a subspace
of $\mR^{n+1}$ and maps it onto the upper hemisphere of a sphere $S$ tangent
to it from below, by a projection centered in the center of $S$. The group of affine isometries
then acts analytically on the closure of the upper hemisphere with two orbits,
the image of $\mR^n$ and the equator. This example is again a compactification
of a symmetric space. Note that there is no hope
of obtaining other examples with non-positively curved symmetric spaces, since non-Euclidean higher-rank spaces
admit no such differentiable compactification \cite{Kloeckner3}.
One can also consider the action of translations on the Poincar\'e compactification of $\mR^n$
to obtain an almost homogeneous manifold with non-homogeneous boundary, satisfying condition (A).

%%%%%%%%%%%%%%%%%
\subsection{Common regularity of semi-conjugacies}

The core result of this section is the following simple lemma.
\begin{lemma}[common regularity]
Assume that (A) holds, and let $\rho_1$ and $\rho_2$ be two $\diffb{r}$ actions of $G$
on $M$, both topologically conjugate to $\rho_0$. Let $\Phi_0$ and $\Phi$ be two homeomorphisms
of $M$ that are $(\rho_1,\rho_2)$ equivariant. 
Then $\Phi$ is $\diffb{r}$ if and only if $\Phi_0$ is.
\end{lemma}

\begin{proof}
Let $x$ be any point of $\Int(M)$ and $H$ its stabilizer for $\rho_1$. 

Since $\rho_0$ is transitive on $\Int(M)$,
so is $\rho_2$ and there is a $g_0\in G$ such that $\Phi(x)=\rho_2(g_0)\Phi_0(x)$.
But by equivariance, the stabilizer for $\rho_2$ of both $\Phi(x)$ and $\Phi_0(x)$
must be $H$. Therefore, $g_0\in\stab(H)$ and from (A) it follows that $g_0=h_0z_0$ where $h_0\in H$ and $z_0\in Z(G)$.

We get for all $g\in G$:
\begin{eqnarray*}
\Phi(\rho_1(g)x) &=& \rho_2(g)\Phi(x) \\
                 &=& \rho_2(g)\rho_2(z_0)\Phi_0(\rho_1(h_0)x) \\
                 &=& \rho_2(z_0)\rho_2(g)\Phi_0(x)\\
                 &=& \rho_2(z_0)\Phi_0(\rho_1(g)x).
\end{eqnarray*}
that is, $\Phi=\rho_2(z_0)\Phi_0$ on $\Int(M)$. By continuity, this equality
holds on the whole of $M$ and since $\rho_2(z_0)$ is a $\diffb{r}$ diffeomorphism,
$\Phi$ is $\diffb{r}$ if and only if $\Phi_0$ is.
\end{proof}

The first consequence of the common regularity lemma is that two stretching
 functions that are not equivalent lead to non-conjugate stretches.
\begin{theorem}\label{nonconjugacy}
If condition (A) holds, then the two stretches of $\rho_0$ associated to functions $f_1$
and $f_2$ are $\diffb{r}$ conjugate  only if $f_1$ and $f_2$ are equivalent as stretching
functions.
\end{theorem}

\begin{proof}
The map $\Phi_0=\Phi_{f_2^{-1}f_1}$ topologically conjugates the two stretches, denoted by
$\rho_1$ and $\rho_2$. Moreover, it is not a conjugacy unless $f_1$ and $f_2$ are equivalent
(read in charts near the boundary).

Let $\Phi:M\to M$ be any $(\rho_1,\rho_2)$ equivariant homeomorphism. 
The common regularity lemma,
applied to $\Phi$ and $\Phi_0$ or to $\Phi^{-1}$ and $\Phi_0^{-1}$, implies that $\Phi$ is
not a $\diffb{r}$ diffeomorphism unless $\Phi_0$ is. Therefore, either $f_1$ and $f_2$
are equivalent or there exists no $\diffb{r}$ conjugacy between $\rho_1$ and $\rho_2$.
\end{proof}

As a striking consequence, we get the following.
\begin{corollary}\label{infinity}
If condition (A) holds, then $\moduli^\omega(\rho_0)$ is at least countably infinite.
\end{corollary}
We will see in the next section some examples where $\moduli^\omega(\rho_0)$ is countable.
It would be interesting to determine if there exist actions $\rho_0$ for which it is
uncountable.

%%%%%%%%%%%%%%%%%
\subsection{Semi-conjugacy as an ordering}

The common regularity lemma can be used to define a natural partial order on
$\moduli^r(\rho_0)$.
\begin{definition}
Let $[\rho_1],[\rho_2]$ be two elements of $\moduli^r(\rho_0)$. We say that
$[\rho_2]$ is \defini{tighter} than $[\rho_1]$ and we write $[\rho_1]\succ^r[\rho_2]$
(sometimes forgetting the brackets or the $r$) if  $\rho_1$ is $\diffb{r}$
semi-conjugate to $\rho_2$.
\end{definition}

Of course, this definition is consistent: the relation holds or not regardless of the
choice of a representative in each conjugacy class.

\begin{proposition}\label{ordering}
If (A) holds, then the relation $\succ^r$ defines a partial order on $\moduli^r(\rho_0)$.
\end{proposition}

\begin{proof}
The reflexivity and transitivity of $\succ$ are obvious. Let us show
that $\succ$ is antisymmetric.

Let $[\rho_1]$ and $[\rho_2]$ be elements of $\moduli^r(\rho_0)$ that is,
$\rho_i$ are $\diffb{r}$ actions of $G$ on $M$ that are topologically conjugate to $\rho_0$.
Assume that $[\rho_1]\succ [\rho_2]$ and $[\rho_2]\succ [\rho_1]$. Then there are two
homeomorphisms $\Phi_a$ and $\Phi_b$ of $M$ that are respectively $(\rho_1,\rho_2)$
and $(\rho_2,\rho_1)$ equivariant and are both $\diffb{r}$ maps. But $\Phi_b^{-1}$
is a $(\rho_1,\rho_2)$ equivariant homeomorphism of $M$, and the common regularity lemma
implies that it is $\diffb{r}$. It therefore defines a $\diffb{r}$ conjugacy
between $\rho_1$ and $\rho_2$.
\end{proof}

Given the action $\rho_0$, an interesting question is whether there is a tightest
element in $\moduli^r(\rho_0)$. It would be the ``fundamental action'' to which every other
is semi-conjugate. In the case of the differentiable compactifications of the
 hyperbolic spaces, we proved in \cite{Kloeckner2} the existence of such a tightest compactification.
Moreover, every other compactification is not only semiconjugate to it, but is a stretch
of it. In the next section, we generalize this result to the other non-positively curved
symmetric spaces of rank one.

%%%%%%%%%%%%%%%%%%%%%%%%%%%%%%%%%%%%%%%%%%%%%%%%%%%%%%%%%%%%%%%
\section{Non-positively curved symmetric spaces of rank one}

%%%%%%%%%%%%%%%%%
\subsection{Differentiable compactifications}

In this subsection we consider a homogeneous $\diffb{r}$ manifold $X=G/H$, where $G$ is connected.
We denote by $\rho$ the corresponding transitive action. 
\begin{definition}\label{diff_comp}
A $\diffb{r}$ \defini{differentiable compactification} of $X$ is the data of
a manifold with boundary $M$ and a $\diffb{r}$ embedding $\psi:X\to M$ such that:
\begin{enumerate}
\item $\psi(X)=\Int M$,
\item the action $\psi_*\rho$ of $G$ on $\Int M$ extends to a $\diffb{r}$
      action on $M$, which is also denoted by $\psi_*\rho$ when no confusion is possible.
\end{enumerate}
\end{definition}

\begin{definition}
Let $(\psi_1,M_1)$ and $(\psi_2,M_2)$ be two differentiable compactifications 
of $X$ and denote $\rho_i$ ($i=1,2$) the extension of $\psi_{i*}\rho$ to $M_i$.
Then $(\psi_1,M_1)$ and $(\psi_2,M_2)$ are
said to be \defini{equivalent} if there are $\diffb{r}$ diffeomorphisms
$\alpha:X\to X$ and $\beta:M_1\to M_2$ that are respectively $\rho$ and 
$(\rho_1,\rho_2)$-equivariant and such that $\psi_2\,\alpha=\beta\,\psi_1$.
\end{definition}

The introduction of $\alpha$ is natural: a mere change of coordinates on $X$ must not
change the equivalence class of a differentiable compactification. But the
condition $\psi_2\,\alpha=\beta\,\psi_1$ entirely defines $\alpha$ and
the equivariance of $\beta$ implies that of $\alpha$. As a consequence,
two differentiable compactifications of $X$ are equivalent if and only if
the actions $\rho_1$ and $\rho_2$ are conjugate.

%%%%%%%%%%%%%%%%%
\subsection{Recapitulation on rank one symmetric spaces}

The classical reference on negatively curved symmetric spaces is \cite{Mostow}.
See also \cite{Epstein,Goldman} for the complex case, and \cite{Allcock,Salzmann}
for the octonionic case.

Let $\mK$ be either $\mC$, $\mH$ or the nonassociative field $\mO$ and let $k=2,4$ or $8$ be the
dimension of $\mK$ as a real algebra.  Denote
by $\KH^m$ (where $m=2$ if $\mK=\mO$) the $\mK$ hyperbolic space,
$G$ the neutral component of its isometry group and $\rho$ the isometric action of $G$ on $\KH^m$.
Denote by $K$ a maximal compact subgroup of $G$; it is the stabilizer of some point
$x_0$ of $\KH^m$. Denote by $n=km$ the real dimension of $\KH^m$.

The action $\rho$ is transitive and, $\KH^n$ being of rank one, isotropic.
As a consequence no two points
of $\KH^m$ have the same stabilizer and $\rho$ satisfies condition
(A).

As a homogeneous space, $\KH^m$ can be identified with an open ball in the projective space $\KP^m$ in such a way
that $\rho$ extends to a real-analytic projective action. The boundary of $\KH^m$ in
$\KP^m$ can be canonically identified with its geodesic boundary $\partial\KH^m$,
defined as the
space of asymptote classes of geodesic rays. The group $K$
act transitively on this boundary.

As a consequence, this embedding into $\KP^m$ gives a $\diffb{\omega}$ 
differentiable compactification $\KH^m\to \overline{\KH}^m$; the
corresponding action of $G$ will be denoted by $\overline{\rho}$ and is called
the projective compactification.

The intersection of any $\mK$-projective lines with $\KH^m$ is a totally geodesic
submanifold of $\KH^n$ isometric to $\RH^k$ (up to a constant that makes it of curvature $-4$).
 These are called the $\mK$-lines of
$\KH^m$, and $G$ acts transitively on the set of $\mK$-lines.
The geodesic boundary of
a $\mK$-line is a properly embedded $(k-1)$-dimensional sphere in $\partial\KH^m$.

Let $\ell$ be a $\mK$-line: the geodesic symmetry around $\ell$ is then a direct isometry.
Moreover, any direct isometry of $\ell$ extends to a direct isometry of the whole
of $\KH^m$, thus the group of direct isometries of $\ell$ can be considered a
subgroup $G_\ell$ of $G$. The restriction of $\overline{\rho}$ to $\overline{\ell}$ and
$G_\ell$ is analytically conjugate to the conformal compactification (see \cite{Kloeckner2}
and below) of $\RH^k$.

Each geodesic is also contained in a totally real $\RH^m$, which is 
a totally geodesic $m$-dimensional submanifold isometric to $\RH^m$ (with curvature $-1$),
whose tangent space at each point is totally real. A totally real
$\RH^m$ is obtained by moving the canonical embedding $\RP^m\to\KP^m$ by
an element of $G$. Any direct isometry of $\lambda$, a totally real $\RH^n$,  extends to 
an element of $G$, therefore the group of direct isometries of $\lambda$ can be considered as a
subgroup $G_\lambda$ of $G$. The restriction of $\overline{\rho}$ to $\overline{\lambda}$ and
$G_\lambda$ is analytically conjugate to the projective compactification (see \cite{Kloeckner2}
and below) of $\RH^k$.

%%%
\subsection{The projective and conformal compactifications of the real hyperbolic space}

The projective compactification of $\RH^n$ is defined like that of $\KH^n$: it is
the restriction to the closure of Klein's ball $\overline{\RH}^n$ of the projective action 
of $\SOo{1,n}$ on $\RP^n$. In this model geodesics are affine lines, and it plays
a central r\^ole in the proof that every $\diffb{r}$ differentiable compactification
of $\RH^n$ is a stretch of the projective one, provided $n>2$.

The conformal compactification is the continuous extension to the closed ball
of the action of $\SOo{1,n}$ on Poincar\'e's ball. In this model, geodesics are
circle arcs orthogonal to the boundary, this orthogonality making sense since
the Euclidean conformal structure on the closed ball is preserved by the action.
As is the case for all compactifications, the conformal compactification is a stretch of the projective one.
The stretching function can be chosen to be $y\mapsto y^2$, and this can be seen
by constructing Poincar\'e's ball from Klein's: one projects the latter
vertically to a hemisphere (here lies the order $2$ term), which is in turn 
stereographically projected to the former.

\begin{figure}[ht]\begin{center}
\includegraphics[width=6cm]{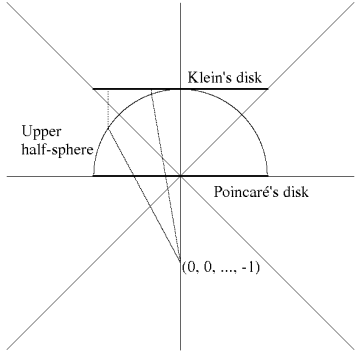}
\caption{Construction of Poincar\'e's ball from Klein's.}
\end{center}\end{figure}

%%%%%%%%%%%%%%%%%
\subsection{Classification of differentiable compactifications}

First note that the following proposition can be proved in the same way as
in the real case (see \cite{Kloeckner2}).
\begin{proposition}
Any differentiable compactification of $\KH^m$ is topologically conjugate to
$\overline{\KH}^m$.
\end{proposition}

As a consequence, we can identify a differentiable compactification and the
homeomorphism that conjugates it to $\overline{\rho}$. The following definition
is thus equivalent to definition \ref{diff_comp} ($\overline{B}^n$ denotes the closed
ball of dimension $n$).
\begin{definition}
By a $\diffb{r}$ \defini{differentiable compactification} of $\KH^m$ we now mean
a homeomorphism $\Phi:\overline{B}^n\to \overline{\KH}^m$ whose restriction to
the interior is a $\diffb{r}$ diffeomorphism such that the
action $\Phi^*(\overline{\rho})$ is $\diffb{r}$ up to the boundary.
\end{definition}

Recall that $r$ is assumed to be $\infty$ or $\omega$.

\begin{theorem}\label{classification}
If $r=\omega$ or $\mK=\mH$ or $\mO$, then
any $\diffb{r}$ compactification of $\KH^m$ is
a stretch of $\overline{\rho}$.
In particular, $(\moduli^\omega(\overline{\rho}),\prec)$ is isomorphic
to $\mN$ endowed with the divisibility ordering.
\end{theorem}

We thus get the very same result as in the real case: the projective compactification
is in each case the tightest differentiable compactification. The same result
may hold in lower regularity, and could be checked by a computation of the action
of $G$ in Klein coordinates. Concerning the $\mK=\mC$, $r=\infty$ case, the result
holds if we ask that the whole group of isometries (and not only its neutral component)
acts smoothly. Otherwise it is open: the following proof relies heavily on the classification
of differentiable compactifications of $\RH^k$, and the $k=2$, $r=\infty$ case is open.

\begin{proof}
Let $\Phi$ be a $\diffb{r}$ compactification of $\KH^m$ and $\rho_1=\Phi^*(\overline{\rho})$.

We fix a point $p\in\KH^m$ and put coordinates $(x,y)\in S_p\KH^m\times[0,\infty)$ on
$\overline{\KH^m}\setminus\{p\}$.

Let $\ell$ be any $\mK-$line of $\KH^m$ through $p$. Let $g$ be the geodesic symmetry
around $\ell$. Then $\rho_1(g)$ has eigenvalues $1$ and $-1$ with multiplicities
$k$ and $mk-k$ at any point of $\Phi^*\ell$, thus also at points of $\overline{\Phi^*\ell}$.
By the implicit function theorem, $\overline{\Phi^*\ell}$ is thus a $\diffb{r}$ submanifold
of $\overline{B}^n$ that is transverse to the boundary.
The restriction of $\rho_1$ to $\overline{\Phi^*\ell}$ and
$G_\ell$ is then a $\diffb{r}$ compactification of $\RH^k$. Due to the classification
of such compactifications and since $\overline{\ell}$ is
the conformal compactification of $\RH^k$, the restriction of $\Phi$ to $\overline{\Phi^*\ell}$
is almost a stretch: in coordinates $(x,y)\in S_p\ell\times[0,1)$ it can be written as
$(x,y)\mapsto (x,\sqrt{f(y)})$ where $f$ is a stretching map.

But by equivariance $\Phi$ can be written in this form on a neighborhood of the whole boundary
(just let the stabilizer $K$ of $p$ act).
If we look at the restriction of $\Phi$ to the closure of a totally real line $\overline{\lambda}$,
it can be written as $(x,y)\mapsto (x,\sqrt{f(y)})$ in coordinates $(x,y)\in S_p\lambda\times [0,1)$ and
defines a $\diffb{r}$ compactification of $\RH^m$. Due to the classification in the real case, 
$\sqrt{f}$ itself must be a stretching function, and $\Phi$ is a stretching map.

Note that we do not need $m\geq 3$ since
$\overline{\lambda}$ is necessarily transverse to the boundary, and so are its geodesics.
\end{proof}

In the case of $\CH^m$, we expect the projective compactification to be the only one to preserve
a complex structure. More generally, one can ask which compactifications of a rank one
symmetric space preserve any geometric structure at all (with Cartan's definition of
a geometric structure for example).

%%%%%%%%%%%%%%%%%%%%%%%%%%%%%%%%%%%%%%%%%%%%%%%%%%%%%%%%%%%%%%%
\section{Prospectus}

Let us stress some limitations of this work that lead to interesting questions.

First, we would like to get rid of condition (A). Without it,
can two non-equivalent stretching maps lead to equivalent actions ? The
existence of a dense open orbit will of course be of primary importance.

Second, the notion of stretch could be used when $M$ has no boundary but a
dense open orbit whose complement is a $1$-codimensional submanifold, but does not
extend as it is to greater codimension. Could one modify it so that the dimension
of the closed orbit does not matter?

Third, in most cases we are only able to construct new actions from a given one.
Given two subgroups $P,K$ in $G$, we would like to determine the (possibly empty) set
of differentiable action of $G$ that have two orbits, with respective stabilizers
$P$ and $K$.

More generally, could we describe all actions of a given Lie groups that have a finite
number of orbits by explicit combinatorial and analytic data?

%%%
\subsubsection*{Acknowledgement}

I wish to thank professors Demailly and Forstneri\v c for useful discussions and
correspondence about the uniqueness of the analytic structure of
a manifold with boundary.

\bibliographystyle{plain}
\bibliography{biblio.bib}

\begin{thebibliography}{10}

\bibitem{Allcock}
Daniel Allcock.
\newblock Reflection groups on the octave hyperbolic plane.
\newblock {\em J. Algebra}, 213(2):467--498, 1999.

\bibitem{Arnold}
Vladimir~I. Arnold.
\newblock {\em Ordinary differential equations}.
\newblock Universitext. Springer-Verlag, Berlin, 2006.
\newblock Translated from the Russian by Roger Cooke, Second printing of the
  1992 edition.

\bibitem{Epstein}
D.~B.~A. Epstein.
\newblock Complex hyperbolic geometry.
\newblock In {\em Analytical and geometric aspects of hyperbolic space
  (Coventry/Durham, 1984)}, volume 111 of {\em London Math. Soc. Lecture Note
  Ser.}, pages 93--111. Cambridge Univ. Press, Cambridge, 1987.

\bibitem{Goldman}
William~M. Goldman.
\newblock {\em Complex hyperbolic geometry}.
\newblock Oxford Mathematical Monographs. The Clarendon Press Oxford University
  Press, New York, 1999.
\newblock Oxford Science Publications.

\bibitem{Grauert}
Hans Grauert.
\newblock On {L}evi's problem and the imbedding of real-analytic manifolds.
\newblock {\em Ann. of Math. (2)}, 68:460--472, 1958.

\bibitem{Kloeckner3}
Beno{\^{\i}}t Kloeckner.
\newblock Symmetric spaces of higher rank do not admit differentiable
  compactifications, 2005.
\newblock preprint.

\bibitem{Kloeckner2}
Beno{\^{\i}}t Kloeckner.
\newblock On differentiable compactifications of the hyperbolic space.
\newblock {\em Transform. Groups}, 11(2):185--194, 2006.

\bibitem{Krantz-Parks}
Steven~G. Krantz and Harold~R. Parks.
\newblock {\em A primer of real analytic functions}.
\newblock Birkh\"auser Advanced Texts: Basler Lehrb\"ucher. [Birkh\"auser
  Advanced Texts: Basel Textbooks]. Birkh\"auser Boston Inc., Boston, MA,
  second edition, 2002.

\bibitem{Mostow}
G.~D. Mostow.
\newblock {\em Strong rigidity of locally symmetric spaces}.
\newblock Princeton University Press, Princeton, N.J., 1973.
\newblock Annals of Mathematics Studies, No. 78.

\bibitem{Salzmann}
Helmut Salzmann, Dieter Betten, Theo Grundh{\"o}fer, Hermann H{\"a}hl, Rainer
  L{\"o}wen, and Markus Stroppel.
\newblock {\em Compact projective planes}, volume~21 of {\em de Gruyter
  Expositions in Mathematics}.
\newblock Walter de Gruyter \& Co., Berlin, 1995.
\newblock With an introduction to octonion geometry.

\bibitem{Tognoli}
Alberto Tognoli.
\newblock Approximation theorems and {N}ash conjecture.
\newblock In {\em Journ\'ees de g\'eom\'etrie analytique (Univ. Poitiers,
  Poitiers, 1972)}, pages 53--68. Bull. Soc. Math. France Suppl. M\'em., No.
  38. Soc. Math. France, Paris, 1974.

\end{thebibliography}

%%%%%%%%%%%%%%%%%%%%%%%%%%%%%%%%%%%%%%%%%%%%%%%%%%%%%%%%%%%%%%%
%%%%%%%%%%%%%%%%%%%%%%%%%%%%%%%%%%%%%%%%%%%%%%%%%%%%%%%%%%%%%%%
\end{document}